\documentclass[12pt]{amsart}
\usepackage{amsmath,mathtools,amsthm,amsfonts,bigints,amssymb, mathrsfs, tikz-cd, xcolor}

\newtheorem{theorem}{Theorem}[section]
\newtheorem{lemma}[theorem]{Lemma}

\newtheorem{corollary}[theorem]{Corollary}

\theoremstyle{definition}

\newtheorem{remark}[theorem]{Remark}

\newcommand{\R}{\mathbb R}%
\newcommand{\C}{\mathbb C}%
\newcommand{\N}{\mathbb N}%
\newcommand{\z}{\mathfrak z}%
\newcommand{\mv}{\mathfrak v}%
\newcommand{\J}{\mathscr J}%
\numberwithin{equation}{section}
\makeatletter

\renewcommand\subsubsection{\@secnumfont}{\bfseries}%
\renewcommand\subsubsection{\@startsection{subsubsection}{3}
  \z@{.5\linespacing\@plus.7\linespacing}{-.5em}%
  {\normalfont\bfseries}}

  \makeatother
\makeatletter
\@namedef{subjclassname@2020}{%
  \textup{2020} Mathematics Subject Classification}
\makeatother

\begin{document}

\title[Regularity and pointwise convergence]{Regularity and pointwise convergence of solutions of the Schr\"odinger operator with radial initial data on Damek-Ricci spaces }

\author[Utsav Dewan]{Utsav Dewan}
\address{Stat-Math Unit, Indian Statistical Institute, 203 B. T. Rd., Kolkata 700108, India}
\email{utsav\_r@isical.ac.in}



\begin{abstract}
One of the most celebrated problems in Euclidean Harmonic analysis is the Carleson's problem: determining the optimal regularity of the initial condition $f$ of the
Schr\"odinger equation given by
\begin{equation*}
\begin{cases}
	 i\frac{\partial u}{\partial t} =\Delta u\:,\:  (x,t) \in \mathbb{R}^n \times \mathbb{R} \\
	u(0,\cdot)=f\:, \text{ on } \mathbb{R}^n \:,
	\end{cases}
\end{equation*}
in terms of the index $\alpha$ such that $f$ belongs to the inhomogeneous Sobolev space $H^\alpha(\mathbb{R}^n)$ , so that the solution of the Schr\"odinger operator $u$ converges pointwise to $f$, $\displaystyle\lim_{t \to 0+} u(x,t)=f(x)$, almost everywhere. In this article, we consider the Carleson's problem for the Schr\"odinger equation with radial initial data on Damek-Ricci spaces and obtain the sharp bound up to the endpoint $\alpha \ge 1/4$, which agrees with the classical Euclidean case.
\end{abstract}

\maketitle
\tableofcontents

\section{Introduction}

The story of this paper begins in 1980, when Carleson \cite{C} posed and solved the problem of determining the amount of regularity required on the initial data $f$, so that the solution of the one dimensional Schr\"odinger equation
\begin{equation*}
\begin{cases}
	 i\frac{\partial u}{\partial t} =\frac{\partial^2 u}{\partial x^2} \:,\:  (x,t) \in \R \times \R \\
	u(0,\cdot)=f\:,\: \text{ on } \R \:,
	\end{cases}
\end{equation*}
converges pointwise to $f$, that is,
\begin{equation*}
\displaystyle\lim_{t \to 0+}u(t,x)=f(x)\:,
\end{equation*}
for almost every $x \in \R$. More precisely, Carleson proved that the pointwise convergence does hold true if $f$ is compactly supported and H\"older continuous of order $> 1/4$. On the other hand, Carleson also showed that there exists a compactly supported function $f$ on $\R$ which is H\"older continuous of order $<1/8$, for which the pointwise limit of the solution diverges almost everywhere. In 1982, Dahlberg and Kenig \cite{DK} noted that the only property of compactly supported H\"older continuous functions of order $> 1/4$, used by Carleson is that these functions belong to the inhomogeneous Sobolev space $H^{1/4}(\R)$, where
\begin{equation*}
H^\alpha(\R):= \left\{f \in L^2(\R): \int_\R {\left(1+{|\xi|}^2\right)}^\alpha\: {\left|\hat{f}(\xi)\right|}^2\: d\xi < \infty\right\}\:.
\end{equation*}
Then they showed that the result of Carleson is actually sharp by constructing a counter-example of an $H^\alpha(\R)$ function for $\alpha < 1/4$, for which the pointwise limit of the solution of the Schr\"odinger equation diverges almost everywhere. Subsequently they posed the question in higher dimensions, to determine the regularity of the initial condition in terms of its Sobolev index, so that the solution of the Schr\"odinger operator converges pointwise to its initial data almost everywhere. This is the celebrated Carleson's problem and for the last five decades, it has stimulated extensive research.

\medskip

In 1983, Cowling \cite{Cowling} studied the problem for a general class of self-adjoint operators on $L^2(X)$ for a measure space $X$  and obtained $\alpha > 1$, to be a sufficient condition for the associated Schr\"odinger operator. For $\R^n$, in 1987, Sj\"olin \cite{Sjolin} improved the bound to $\alpha > 1/2$, by means of a local smoothing effect. This improvement was also independently obtained by Vega \cite{Vega} in 1988. In 1990, Prestini \cite{Prestini} solved the Carleson's problem while specializing to radial functions and obtained the dimension independent sharp bound $\alpha \ge 1/4$. After continuous improvements made by several experts in the field, recently in 2017, the sharp bound $\alpha > 1/3$, up to the endpoint was obtained for $\R^2$ by Du-Guth-Li  \cite{DGL} using polynomial partitioning and decoupling. Then in 2019, Du-Zhang \cite{DZ} obtained the bound $\alpha > n/2(n+1)$, for dimensions $n \ge 3$, by using decoupling and induction on scales. This bound is sharp up to the endpoint due to a counterexample by Bourgain \cite{Bourgain}. Thus in the Euclidean setting, the Carleson's problem has been fully resolved.

\medskip

In the case of a Riemannian manifold, where the volume measure neither satisfies any doubling property nor does it admit dilations, the above mentioned Euclidean machineries work no longer and thus it offers a fresh challenge. Recently in this regard, Wang and Zhang considered the problem in the non-Euclidean setting of the Real Hyperbolic spaces and proved that $\alpha > 1/2$, is a sufficient condition for the pointwise convergence to hold true \cite[Theorem 1.1]{WZ}. Our aim in this paper is to show that the above regularity condition can be improved from $\alpha > 1/2$ to the sharp bound up to the endpoint $\alpha \ge 1/4$, while specializing to radial initial data, on the vastly general setting of Damek-Ricci spaces, generalizing Prestini's result in $\R^n$ \cite{Prestini}.

\medskip

Real Hyperbolic spaces are the simplest examples of rank one Riemannian Symmetric spaces of noncompact type. On the other hand, the class of rank one Riemannian Symmetric spaces of noncompact type (barring the degenerate case of the Real Hyperbolic spaces) is contained in (and in fact accounts for a very small subclass of) the more general class of Damek-Ricci spaces \cite{ADY}. They are also known as Harmonic $NA$ groups. These spaces $S$ are non-unimodular, solvable extensions of Heisenberg type groups $N$, obtained by letting $A=\R^+$ act on $N$ by homogeneous dilations. We refer to Section 2 for more details about their structure and analysis thereon.

\medskip

Let $\Delta$ be the Laplace-Beltrami operator
on $S$ corresponding to the left-invariant Riemannian metric. Its $L^2$-spectrum is the half line $(-\infty,  -Q^2/4]$, where $Q$ is the homogeneous dimension of $N$. The Schr\"odinger equation on $S$ is given by
\begin{equation} \label{schrodinger}
\begin{cases}
	 i\frac{\partial u}{\partial t} =\Delta u\:,\:  (x,t) \in S \times \R \\
	u(0,\cdot)=f\:,\: \text{ on } S \:.
	\end{cases}
\end{equation}

\medskip

Then for a radial function $f$ belonging to the $L^2$-Schwartz class (\cite{Anker}, \cite[p. 652]{ADY}),
\begin{equation} \label{schrodinger_soln}
S_t f(x):= \int_{0}^\infty \varphi_\lambda(x)\:e^{it\left(\lambda^2 + \frac{Q^2}{4}\right)}\:\hat{f}(\lambda)\: {|{\bf c}(\lambda)|}^{-2}\: d\lambda\:,
\end{equation}
is the solution to (\ref{schrodinger}), where $\varphi_\lambda$ are the spherical functions, $\hat{f}$ is the Spherical Fourier transform of $f$ and ${\bf c}(\cdot)$ denotes the Harish-Chandra's ${\bf c}$-function. Indeed in this case, by a simple application of the Dominated Convergence Theorem, it follows that for all $x \in S$,
\begin{equation*}
\displaystyle\lim_{t \to 0+} S_tf(x)=f(x)\:.
\end{equation*}

\medskip

Thus one is naturally interested in the Carleson's problem in the present setting of Damek-Ricci spaces. We quantify the problem by defining the Sobolev spaces on $S$ \cite{APV}:
\begin{equation} \label{sobolev_space_defn}
H^\alpha(S):=\left\{f \in L^2(S): {\|f\|}_{H^\alpha(S)}:= {\left(\int_0^\infty {\left(\lambda^2 + \frac{Q^2}{4}\right)}^\alpha {|\hat{f}(\lambda)|}^2 {|{\bf c}(\lambda)|}^{-2} d\lambda\right)}^{1/2}< \infty\right\}.
\end{equation}

\medskip

In this regard, we define the associated maximal function,
\begin{equation} \label{maximal_fn_defn}
S^* f(x):= \displaystyle\sup_{0<t<4/Q^2} \left|S_tf(x)\right|\:.
\end{equation}
Our main result is the following $L^1_{loc}$ boundedness for the maximal function (see section $2$ for the definitions of $m_\mv$ and $m_\z$):
\begin{theorem} \label{maximal_bddness_thm}
Let $f$ be a radial $L^2$-Schwartz class function on $S$ and let $B_R$ denote the geodesic ball centered at the identity with radius $R>0$. Then
\begin{equation} \label{maximal_bddness_inequality}
{\|S^*f\|}_{L^1\left(B_R\right)} \le c\: {\|f\|}_{H^\alpha(S)}\:,
\end{equation}
holds for some constant $c>0$, depending only on the dimensions $m_\mv,m_\z$ and radius $R$, for all $R>0$ and $\alpha \ge 1/4$.
\end{theorem}

\medskip

Then by standard arguments in the literature (for instance see the proof of Theorem 5 of \cite{Sjolin}), we get the positive result on pointwise convergence:
\begin{corollary} \label{pointwise_conv_cor}
The solution $u(x,t)$ of (\ref{schrodinger}) converges pointwise to the radial initial data $f$,
\begin{equation*}
\displaystyle\lim_{t \to 0+} u(x,t)=f(x)\:,
\end{equation*}
for almost every $x$ in $S$ with respect to the left Haar measure on $S$, whenever $f \in H^\alpha(S)$ with $\alpha \ge 1/4$.
\end{corollary}

\begin{remark}
As $N$ has been assumed to be noncommutative, the class of Damek-Ricci spaces does not include the degenerate case of the Real Hyperbolic spaces. But a curious reader may observe that the arguments presented in this article can be carried out verbatim to obtain analogues of Theorem \ref{maximal_bddness_thm} and Corollary \ref{pointwise_conv_cor} for the Real Hyperbolic spaces.
\end{remark}

In this regard, we have the following which illustrates the sharpness of the above results:
\begin{theorem} \label{sharpness_thm}
For the 3-dimensional Real Hyperbolic space, $\mathbb{H}^3 \cong SL(2,\C)/SU(2)$,  the inequality (\ref{maximal_bddness_inequality}) fails if $\alpha < 1/4$.
\end{theorem}

In \cite{WZ}, the authors obtain the result on the Real Hyperbolic spaces by means of a local smoothing effect, building up on Doi's ideas \cite{Doi} on the interaction between smoothing effects of the Schr\"odinger evolution group and the behaviour of the geodesic flow on connected, simply connected, complete Riemannian manifolds of constant non-positive sectional curvature. The main idea in our article however is completely different and is in fact quite simple: keep track of the oscillation afforded by both the spherical functions as well as the multiplier corresponding to the Schr\"odinger operator to improve the bound from $\alpha > 1/2$, to $\alpha \ge 1/4$. Our approach is motivated by the work of Prestini \cite{Prestini}. However, unlike in the Euclidean case, the spherical function on a Damek-Ricci space is not merely a single Bessel function. In fact, it only admits certain series expansions (the Bessel series expansion and a series expansion similar to the classical Harish-Chandra series expansion) depending on the geodesic distance from the identity of the group. Then to capture the aforementioned oscillation from these expansions, to decompose the linearized maximal function into suitable pieces and then to estimate each of them individually, become somewhat technical. In this regard, we also need the improved estimates obtained by Anker-Pierfelice-Vallarino on the coefficients of the above series expansion of the spherical functions (see (\ref{coefficient_estimate})). In order to get the counter-example for lower regularity, we make use of the Abel transform on Ch\'ebli-Trim\`eche Hypergroups \cite{BX}. We also indicate about analogues of Theorem \ref{maximal_bddness_thm}, Corollary \ref{pointwise_conv_cor} and Theorem \ref{sharpness_thm} for the fractional Schr\"odinger equation (see subsection $6.1$).

\medskip

This article is organized as follows. In section $2$, we recall the essential preliminaries about Damek-Ricci spaces and Spherical Fourier Analysis on these spaces. In section $3$, we obtain estimates of an oscillatory integral which is the heart of the matter. In section $4$, we present the proof of Theorem \ref{maximal_bddness_thm}. In section $5$, we prove Theorem \ref{sharpness_thm}. Finally, in section $6$, we conclude by making some remarks and posing some new problems.

\medskip

Throughout, the symbols `c' and `C' will denote positive constants whose values may change on each occurrence. $\N$ will denote the set of positive integers. We will also use the following notation: $\frac{1}{2}\N= \left\{\frac{n}{2}: n \in \N \cup \{0\}\right\}$. Two non-negative functions $f_1$ and $f_2$ will be said to satisfy $f_1 \asymp f_2$ if there exists a constant $C \ge 1$, so that
\begin{equation*}
\frac{1}{C} f_1 \le f_2 \le Cf_1\:.
\end{equation*}

\section{Preliminaries}
In this section, we will explain the notations and state relevant results on Damek-Ricci spaces. Most of these results can be found in \cite{ADY, APV, A}.

\medskip

Let $\mathfrak n$ be a two-step real nilpotent Lie algebra equipped with an inner product $\langle, \rangle$. Let $\mathfrak{z}$ be the center of $\mathfrak n$ and $\mathfrak v$ its orthogonal complement. We say that $\mathfrak n$ is an $H$-type algebra if for every $Z\in \mathfrak z$ the map $J_Z: \mathfrak v \to \mathfrak v$ defined by
\begin{equation*}
\langle J_z X, Y \rangle = \langle [X, Y], Z \rangle, \:\:\:\: X, Y \in \mathfrak v
\end{equation*}
satisfies the condition $J_Z^2 = -|Z|^2I_{\mathfrak v}$, $I_{\mathfrak v}$ being the identity operator on $\mathfrak v$. A connected and simply connected Lie group $N$ is called an $H$-type group if its Lie algebra is $H$-type. Since $\mathfrak n$ is nilpotent, the exponential map is a diffeomorphism
and hence we can parametrize the elements in $N = \exp \mathfrak n$ by $(X, Z)$, for $X\in \mathfrak v, Z\in \mathfrak z$. It follows from the Baker-Campbell-Hausdorff formula that the group law in $N$ is given by
\begin{equation*}
\left(X, Z \right) \left(X', Z' \right) = \left(X+X', Z+Z'+ \frac{1}{2} [X, X']\right), \:\:\:\: X, X'\in \mathfrak v; ~ Z, Z'\in \mathfrak z.
\end{equation*}
The group $A = \R^+$ acts on an $H$-type group $N$ by nonisotropic dilation: $(X, Z) \mapsto (\sqrt{a}X, aZ)$. Let $S = NA$ be the semidirect product of $N$ and $A$ under the above action. Thus the multiplication in $S$ is given by
\begin{equation*}
\left(X, Z, a\right)\left(X', Z', a'\right) = \left(X+\sqrt aX', Z+aZ'+ \frac{\sqrt a}{2} [X, X'], aa' \right),
\end{equation*}
for $X, X'\in \mathfrak v; ~ Z, Z'\in \mathfrak z; a, a' \in \R^+$.
Then $S$ is a solvable, connected and simply connected Lie group having Lie algebra $\mathfrak s = \mathfrak v \oplus \mathfrak z \oplus \R$ with Lie bracket
\begin{equation*}
\left[\left(X, Z, l \right), \left(X', Z', l' \right)\right] = \left(\frac{1}{2}lX' - \frac{1}{2} l'X, lZ'-lZ + [X, X'], 0\right).
\end{equation*}

We write $na = (X, Z, a)$ for the element $\exp(X + Z)a, X\in \mathfrak v, Z \in \mathfrak z, a\in A$. We note
that for any $Z \in \mathfrak z$ with $|Z| = 1$, $J_Z^2 = -I_{\mathfrak v}$; that is, $J_Z$ defines a complex structure
on $\mathfrak v$ and hence $\mathfrak v$ is even dimensional. $m_\mv$ and $m_z$ will denote the dimension of $\mv$ and $\z$ respectively. Let $n$ and $Q$ denote dimension and the homogenous dimension of $S$ respectively:
\begin{equation*}
n=m_{\mv}+m_\z+1 \:\: \textit{ and } \:\: Q = \frac{m_\mv}{2} + m_\z.
\end{equation*}

\medskip

The group $S$ is equipped with the left-invariant Riemannian metric induced by
\begin{equation*}
\langle (X,Z,l), (X',Z',l') \rangle = \langle X, X' \rangle + \langle Z, Z' \rangle + ll'
\end{equation*}
on $\mathfrak s$. For $x \in S$, we denote by $s=d(e,x)$, that is, the geodesic distance of $x$ from the identity $e$. Then the left Haar measure $dx$ of the group $S$ may be normalized so that
\begin{equation*}
dx= A(s)\:ds\:d\sigma(\omega)\:,
\end{equation*}
where $A$ is the density function given by,
\begin{equation} \label{density_function}
A(s)= 2^{m_\mv + m_\z} \:{\left(\sinh (s/2)\right)}^{m_\mv + m_\z}\: {\left(\cosh (s/2)\right)}^{m_\z} \:,
\end{equation}
and $d\sigma$ is the surface measure of the unit sphere.

\medskip

A function $f: S \to \C$ is said to be radial if, for all $x$ in $S$, $f(x)$ depends only on the geodesic distance of $x$ from the identity $e$. If $f$ is radial, then
\begin{equation*}
\int_S f(x)~dx=\int_{0}^\infty f(s)~A(s)~ds\:.
\end{equation*}

\medskip

We now recall the spherical functions on Damek-Ricci spaces. The spherical functions $\varphi_\lambda$ on $S$, for $\lambda \in \C$ are the radial eigenfunctions of the Laplace-Beltrami operator $\Delta$, satisfying the following normalization criterion
\begin{equation*}
\begin{cases}
 & \Delta \varphi_\lambda = - \left(\lambda^2 + \frac{Q^2}{4}\right) \varphi_\lambda  \\
& \varphi_\lambda(e)=1 \:.
\end{cases}
\end{equation*}
For all $\lambda \in \R$ and $x \in S$, the spherical functions satisfy
\begin{equation*}
\varphi_\lambda(x)=\varphi_\lambda(s)= \varphi_{-\lambda}(s)\:.
\end{equation*}
It also satisfies for all $\lambda \in \R$ and all $s \ge 0$:
\begin{equation} \label{phi_lambda_bound}
\left|\varphi_\lambda(s)\right| \le 1\:.
\end{equation}

\medskip

The spherical functions are crucial as they help us define the Spherical Fourier transform of a ``nice" radial function $f$ (on $S$) in the following way:
\begin{equation*}
\hat{f}(\lambda):= \int_S f(x) \varphi_\lambda(x) dx = \int_0^\infty f(s) \varphi_\lambda(s) A(s) ds\:.
\end{equation*}
The Harish-Chandra ${\bf c}$-function is defined as
\begin{equation}
{\bf c}(\lambda)= \frac{2^{(Q-2i\lambda)} \Gamma(2i\lambda)}{\Gamma\left(\frac{Q+2i\lambda}{2}\right)} \frac{\Gamma\left(\frac{n}{2}\right)}{\Gamma\left(\frac{m_\mv + 4i\lambda+2}{4}\right)}\:,
\end{equation}
for all $\lambda \in \R$. One has the following inversion formula (when valid) for radial functions:
\begin{equation*}
f(x)= C \int_{0}^\infty \hat{f}(\lambda)\varphi_\lambda(x) {|{\bf c}(\lambda)|}^{-2}
d\lambda\:,
\end{equation*}
where $C$ depends only on $m_\mv$ and $m_\z$. Moreover, the Spherical Fourier transform extends to an isometry from the space of radial $L^2$ functions on $S$ onto $L^2\left((0,\infty),C{|{\bf c}(\lambda)|}^{-2} d\lambda\right)$. The following estimate of the weight function of the Plancherel measure will be needed (see \cite[Lemma 4.8]{RS}):
\begin{equation} \label{plancherel_measure}
{|{\bf c}(\lambda)|}^{-2} \asymp \:{|\lambda|}^2 {\left(1+|\lambda|\right)}^{n-3}\:.
\end{equation}

\medskip

For the purpose of our article, we will require both the Bessel series expansion as well as another series expansion (similar to the Harish-Chandra series expansion) of the spherical functions. We first see an expansion of $\varphi_\lambda$ in terms of Bessel functions for points near the identity. But before that, let us define the following normalizing constant in terms of the Gamma functions,
\begin{equation*}
c_0 = 2^{m_\z}\: \pi^{-1/2}\: \frac{\Gamma(n/2)}{\Gamma((n-1)/2)}\:,
\end{equation*}
and the following functions on $\C$, for all $\mu \ge 0$,
\begin{equation*}
\J_\mu(z)= 2^\mu \: \pi^{1/2} \: \Gamma\left(\mu + \frac{1}{2} \right) \frac{J_\mu(z)}{z^\mu},
\end{equation*}
where $J_\mu$ are the Bessel functions \cite[p. 154]{SW}.
\begin{lemma}\cite[Theorem 3.1]{A} \label{bessel_series_expansion}
There exist $R_0, 2<R_0<2R_1$, such that for any $0 \le s \le R_0$, and any integer $M \ge 0$, and all $\lambda \in \R$, we have
\begin{equation*}
\varphi_\lambda(s)= c_0 {\left(\frac{s^{n-1}}{A(s)}\right)}^{1/2} \displaystyle\sum_{l=0}^M a_l(s)\J_{\frac{n-2}{2}+l}(\lambda s) s^{2l} + E_{M+1}(\lambda,s)\:,
\end{equation*}
where
\begin{equation*}
a_0 \equiv 1\:,\: |a_l(s)| \le C {(4R_1)}^{-l}\:,
\end{equation*}
and the error term has the following behaviour
\begin{equation*}
	\left|E_{M+1}(\lambda,s) \right| \le C_M \begin{cases}
	 s^{2(M+1)}  & \text{ if  }\: |\lambda s| \le 1 \\
	s^{2(M+1)} {|\lambda s|}^{-\left(\frac{n-1}{2} + M +1\right)} &\text{ if  }\: |\lambda s| > 1 \:.
	\end{cases}
\end{equation*}
Moreover, for every $0 \le s <2$, the series
\begin{equation*}
\varphi_\lambda(s)= c_0 {\left(\frac{s^{n-1}}{A(s)}\right)}^{1/2} \displaystyle\sum_{l=0}^\infty a_l(s)\J_{\frac{n-2}{2}+l}(\lambda s) s^{2l}\:,
\end{equation*}
is absolutely convergent.
\end{lemma}

The following result dealing with expansion of Bessel functions in terms of oscillatory terms will also be important for us.
\begin{lemma}\cite[lemma 1]{Prestini} \label{bessel_function_expansion}
Let $\mu \in \frac{1}{2}\N$. Then
there exist positive constants $A_\mu, c_\mu$ such that
\begin{equation*}
J_\mu(s)= \sqrt{\frac{2}{\pi s}} \cos \left( s - \frac{\pi}{2}\mu - \frac{\pi}{4}\right) + \tilde{E}_\mu(s)\:,
\end{equation*}
where
\begin{equation*}
|\tilde{E}_\mu(s)| \le \frac{c_\mu}{s^{3/2}}\:,\:\text{ for } s \ge A_\mu\:.
\end{equation*}
\end{lemma}

For the asymptotic behaviour of the spherical functions when the distance from the identity is large, we look at the following series expansion \cite[pp. 735-736]{APV}:
\begin{equation} \label{harish_chandra_expansion}
\varphi_\lambda(s)= 2^{-m_\z/2} {A(s)}^{-1/2} \left\{{\bf c}(\lambda)  \displaystyle\sum_{\mu=0}^\infty \Gamma_\mu(\lambda) e^{(i \lambda-\mu) s} + {\bf c}(-\lambda) \displaystyle\sum_{\mu=0}^\infty \Gamma_\mu(-\lambda) e^{-(i\lambda + \mu) s}\right\}\:.
\end{equation}
The above series converges for $\lambda \in \R$, uniformly on compacts not containing the group identity, where $\Gamma_0 \equiv 1$ and for $\mu \in \N$, one has the recursion formula,
\begin{equation*}
(\mu^2-2i\mu\lambda) \Gamma_\mu(\lambda) = \displaystyle\sum_{j=0}^{\mu -1}\omega_{\mu -j}\Gamma_{j}(\lambda)\:.
\end{equation*}
Then one has the following estimate on the coefficients \cite[Lemma 1]{APV}, for constants $C>0, d \ge 0$:
\begin{equation} \label{coefficient_estimate}
\left|\Gamma_\mu(\lambda)\right| \le C \mu^d {\left(1+|\lambda|\right)}^{-1}\:,
\end{equation}
for all $\lambda \in \R, \mu \in \N$.

\section{Estimate of an Oscillatory integral}
In this section, we prove an estimate of an oscillatory integral which will be crucial for the proof of Theorem \ref{maximal_bddness_thm}.

\medskip

Let $B(n)= \max\left\{R_0,A_{\frac{n-2}{2}}\right\}$, where $R_0$ and $A_{\frac{n-2}{2}}$ are as in Lemmas \ref{bessel_series_expansion} and \ref{bessel_function_expansion} respectively. Then we note that, in particular  $B(n) >2$. We now present the result:
\begin{lemma}
\label{oscillatory_integral_estimate}
Let $t: (0,\infty) \to (0,4/Q^2)$ be a measurable function and $R>0$. Then for $s,s' \in (0,R]$,
\begin{equation} \label{oscillatory_integral_inequality}
\left|\bigintssss_{\max\left\{\frac{B(n)}{s},\frac{B(n)}{s'}\right\}}^\infty \frac{e^{i\left\{\lambda(s-s')+(t(s)-t(s'))\left(\lambda^2 + \frac{Q^2}{4}\right)\right\}}}{{\left(\lambda^2 + \frac{Q^2}{4}\right)}^{1/4}} d\lambda \right| \le \frac{c}{{|s-s'|}^{1/2}}\:,
\end{equation}
where $c$ is a positive constant depending only on $m_\mv,m_\z$ and $R$.
\end{lemma}
\begin{proof}
Let $I$ denote the left hand side of the inequality (\ref{oscillatory_integral_inequality}). Let $\beta = Q/2$. Then we note that
\begin{equation*}
I= \frac{1}{\sqrt{\beta}} \left|\bigintssss_{\max\left\{\frac{B(n)}{s},\frac{B(n)}{s'}\right\}}^\infty \frac{e^{i\left\{\lambda(s-s')+(t(s)-t(s'))\lambda^2 \right\}}}{{\left(1+ {\left(\frac{\lambda}{\beta}\right)}^2 \right)}^{1/4}} d\lambda \right|\:.
\end{equation*}
Now by the change of variable, $r=\lambda/\beta$, one further gets that
\begin{equation*}
I = \sqrt{\beta} \left|\bigintssss_{\max\left\{\frac{B(n)}{\beta s},\frac{B(n)}{\beta s'}\right\}}^\infty \frac{e^{i\left\{r(\beta s-\beta s')+(t(s)-t(s'))\beta^2r^2 \right\}}}{{\left(1+ r^2 \right)}^{1/4}} dr \right|\:.
\end{equation*}
Now define the $(0,1)$-valued measurable function $\tilde{t}$ on $(0,\infty)$ by
\begin{equation*}
\tilde{t}(u):= \beta^2\: t(u/\beta)\:.
\end{equation*}
Then one rewrites $I$ as
\begin{equation*}
I= \sqrt{\beta} \left|\bigintssss_{\max\left\{\frac{B(n)}{\beta s},\frac{B(n)}{\beta s'}\right\}}^\infty \frac{e^{i\left\{r(\beta s-\beta s')+(\tilde{t}(\beta s)-\tilde{t}(\beta s'))r^2 \right\}}}{{\left(1+ r^2 \right)}^{1/4}} dr \right|\:.
\end{equation*}
Finally substituting $\beta s$ and $\beta s'$ by $s$ and $s'$ respectively, $I$ takes the form
\begin{equation*}
I= \sqrt{\beta} \left|\bigintssss_{\max\left\{\frac{B(n)}{s},\frac{B(n)}{s'}\right\}}^\infty \frac{e^{i\left\{r(s-s')+(\tilde{t}(s)-\tilde{t}(s'))r^2 \right\}}}{{\left(1+ r^2 \right)}^{1/4}} dr \right|\:,
\end{equation*}
where $s$ and $s'$ belong to the interval $(0,\beta R]$. The result now follows by proceeding as in the proof of Lemma 2 of \cite{Prestini}.
\end{proof}

\section{Proof of Theorem \ref{maximal_bddness_thm}
}
We begin the proof by linearizing the maximal function. Then writing $t$ as a $(0,4/Q^2)$-valued radial measurable function of $x$, we get
\begin{equation*}
Tf(s):=S_{t(s)}f(s)= \int_{0}^\infty \varphi_\lambda(s)\:e^{it(s)\left(\lambda^2 + \frac{Q^2}{4}\right)}\:\hat{f}(\lambda)\: {|{\bf c}(\lambda)|}^{-2}\: d\lambda\:.
\end{equation*}
Then it suffices to prove for any $R>0$, the following estimate at the endpoint $\alpha=1/4$,
\begin{equation} \label{thm_pf_eq1}
\int_0^R |Tf(s)|\: A(s)\: ds \le c\: {\|f\|}_{H^{1/4}(S)}\:,
\end{equation}
for some positive constant $c$, depending only on $m_\mv,m_\z$ and $R$.

\medskip

We first decompose $T$ as follows,
\begin{eqnarray*}
Tf(s) &=& \int_{0}^{B(n)/s} \varphi_\lambda(s)\:e^{it(s)\left(\lambda^2 + \frac{Q^2}{4}\right)}\:\hat{f}(\lambda)\: {|{\bf c}(\lambda)|}^{-2}\: d\lambda \\
& + & \int_{B(n)/s}^{\infty} \varphi_\lambda(s)\:e^{it(s)\left(\lambda^2 + \frac{Q^2}{4}\right)}\:\hat{f}(\lambda)\: {|{\bf c}(\lambda)|}^{-2}\: d\lambda \\
&=& T_1f(s)+T_2f(s)\:.
\end{eqnarray*}
In the above, $B(n)$ is as in section $3$.

\subsection{Estimating $T_1$:} In estimating $T_1$, we will only use (\ref{phi_lambda_bound}), that is, boundedness of $\varphi_\lambda$. By an application of the Cauchy-Schwarz inequality, we get that
\begin{eqnarray} \label{thm_pf_eq2}
&& |T_1f(s)| \nonumber\\
&\le &   \int_0^{B(n)/s} \left|\hat{f}(\lambda)\right| {|{\bf c}(\lambda)|}^{-2} d\lambda \nonumber\\
& \le &  {\left(\bigintssss_0^\infty {\left(\lambda^2 + \frac{Q^2}{4}\right)}^{1/4} {\left|\hat{f}(\lambda)\right|}^2 {|{\bf c}(\lambda)|}^{-2} d\lambda \right)}^{1/2}  {\left(\bigintssss_0^{B(n)/s} \frac{{|{\bf c}(\lambda)|}^{-2} d\lambda}{{\left(\lambda^2 + \frac{Q^2}{4}\right)}^{1/4}}\right)}^{1/2} \nonumber\\
&=&  {\|f\|}_{H^{1/4}(S)} {\left(\bigintssss_0^{B(n)/s} \frac{{|{\bf c}(\lambda)|}^{-2} d\lambda}{{\left(\lambda^2 + \frac{Q^2}{4}\right)}^{1/4}}\right)}^{1/2}\:.
\end{eqnarray}
Now by the estimate of ${|{\bf c}(\lambda)|}^{-2}$ (see \ref{plancherel_measure}), one has
for some constant $c$ depending only on $m_\mv$ and $m_\z$,
\begin{equation} \label{thm_pf_eq3}
\bigintssss_0^{B(n)/s} \frac{{|{\bf c}(\lambda)|}^{-2} d\lambda}{{\left(\lambda^2 + \frac{Q^2}{4}\right)}^{1/4}} \le c \left(1+\frac{1}{s^n}\right)\:.
\end{equation}
Then plugging (\ref{thm_pf_eq3}) in (\ref{thm_pf_eq2}) and using the expression of the density function (\ref{density_function}), we get for some constant $c$ depending only on $m_\mv,m_\z$ and $R$, that
\begin{eqnarray*}
\int_0^R \left|T_1f(s)\right| A(s) ds & \le & c\: {\|f\|}_{H^{1/4}(S)} \bigintssss_0^R {\left(1+\frac{1}{s^n}\right)}^{1/2} A(s) ds \\
& \le & c\: {\|f\|}_{H^{1/4}(S)} \left\{\int_0^1 s^{\frac{n}{2}-1}\: ds + \int_1^R e^{Qs}\: ds\right\} \\
& \le & c\: {\|f\|}_{H^{1/4}(S)}\:.
\end{eqnarray*}
Thus (\ref{thm_pf_eq1}) is true for $T_1$.

\subsection{Estimating $T_2$:} The first step of proving (\ref{thm_pf_eq1}) for $T_2$ is to decompose the integral into two parts depending on the geodesic distance from the identity (that is the values of $s$) and then invoke the series expansions of $\varphi_\lambda$. More precisely, we write
\begin{equation*}
\int_0^R \left|T_2f(s)\right| A(s) ds = \int_0^{R_0} \left|T_2f(s)\right| A(s) ds + \int_{R_0}^R \left|T_2f(s)\right| A(s) ds = I_1+I_2\:.
\end{equation*}
In the above, $R_0$ is as in Lemma \ref{bessel_series_expansion}.
\subsubsection{Estimating $I_1$:} The strategy to estimate $I_1$ is to further decompose $T_2$ into three parts. Two of them will be taken care of by estimating the oscillation, while the third one will follow by estimates on the error terms.

\medskip

We expand $\varphi_\lambda$ in the Bessel series. For $0 \le s \le R_0$, putting $M=0$ in Lemma \ref{bessel_series_expansion}, we get for $\lambda \ge B(n)/s$,
\begin{equation} \label{thm_pf_eq4}
\varphi_\lambda(s)= c_0 {\left(\frac{s^{n-1}}{A(s)}\right)}^{1/2} \J_{\frac{n-2}{2}}(\lambda s) + E_1(\lambda,s)\:,
\end{equation}
where
\begin{equation*}
\left|E_1(\lambda,s)\right| \le cs^2 {(\lambda s)}^{-\frac{n+1}{2}}\:.
\end{equation*}
Now by definition,
\begin{equation} \label{thm_pf_eq5}
\J_{\frac{n-2}{2}}(\lambda s)= 2^{\frac{n-2}{2}} \: \pi^{1/2} \: \Gamma\left(\frac{n-1}{2}\right) \frac{J_{\frac{n-2}{2}}(\lambda s)}{(\lambda s)^{\left(\frac{n-2}{2}\right)}}\:.
\end{equation}
Moreover, by Lemma \ref{bessel_function_expansion},
\begin{eqnarray} \label{thm_pf_eq6}
J_{\frac{n-2}{2}}(\lambda s) &=& \sqrt{\frac{2}{\pi}} \frac{1}{{(\lambda s)}^{1/2}} \cos \left(\lambda s - \frac{\pi}{2}\left(\frac{n-2}{2}\right) - \frac{\pi}{4}\right) + \tilde{E}_{\frac{n-2}{2}}(\lambda s) \nonumber \\
&=& \sqrt{\frac{1}{2\pi}} \frac{1}{{(\lambda s)}^{1/2}} \left\{e^{i\left(\lambda s - \frac{\pi}{4}(n-1)\right)} + e^{-i\left(\lambda s - \frac{\pi}{4}(n-1)\right)}\right\} + \tilde{E}_{\frac{n-2}{2}}(\lambda s)\:,
\end{eqnarray}
where
\begin{equation*}
\left|\tilde{E}_{\frac{n-2}{2}}(\lambda s)\right| \le \frac{c}{{(\lambda s)}^{3/2}}\:.
\end{equation*}
Then pluggining (\ref{thm_pf_eq6}) and (\ref{thm_pf_eq5}) in (\ref{thm_pf_eq4}), we get for some $c>0$,
\begin{equation} \label{thm_pf_eq7}
\varphi_\lambda(s)= c {\left(\frac{s^{n-1}}{A(s)}\right)}^{1/2} \frac{1}{{(\lambda s)}^{\frac{n-1}{2}}} \left\{e^{i\left(\lambda s - \frac{\pi}{4}(n-1)\right)} + e^{-i\left(\lambda s - \frac{\pi}{4}(n-1)\right)}\right\} + \mathscr{E}_1(\lambda,s)\:,
\end{equation}
where
\begin{equation*}
\mathscr{E}_1(\lambda,s) = E_1(\lambda,s) + c {\left(\frac{s^{n-1}}{A(s)}\right)}^{1/2} \:\frac{\tilde{E}_{\frac{n-2}{2}}(\lambda s)}{(\lambda s)^{\left(\frac{n-2}{2}\right)}}\:,
\end{equation*}
and thus
\begin{eqnarray} \label{thm_pf_eq8}
\left|\mathscr{E}_1(\lambda,s)\right| & \le & c \left\{s^2 {(\lambda s)}^{-\frac{n+1}{2}} + {(\lambda s)}^{-\frac{n+1}{2}} \right\} \nonumber \\
& \le & c \:{(\lambda s)}^{-\frac{n+1}{2}}\:.
\end{eqnarray}

Thus for $0 \le s \le R_0$ and $\lambda \ge B(n)/s$, using (\ref{thm_pf_eq7}), we decompose $T_2$ as,
\begin{eqnarray*}
&& T_2f(s)\\
&=& c \bigintssss_{B(n)/s}^\infty {\left(\frac{s^{n-1}}{A(s)}\right)}^{1/2} \frac{1}{{(\lambda s)}^{\frac{n-1}{2}}} \: e^{i\left(\lambda s - \frac{\pi}{4}(n-1)\right)}  \:e^{it(s)\left(\lambda^2 + \frac{Q^2}{4}\right)}\:\hat{f}(\lambda)\: {|{\bf c}(\lambda)|}^{-2}\: d\lambda \\
&+& c \bigintssss_{B(n)/s}^\infty {\left(\frac{s^{n-1}}{A(s)}\right)}^{1/2} \frac{1}{{(\lambda s)}^{\frac{n-1}{2}}} \: e^{-i\left(\lambda s - \frac{\pi}{4}(n-1)\right)}  \:e^{it(s)\left(\lambda^2 + \frac{Q^2}{4}\right)}\:\hat{f}(\lambda)\: {|{\bf c}(\lambda)|}^{-2}\: d\lambda \\
&+& \bigintssss_{B(n)/s}^\infty \mathscr{E}_1(\lambda,s) \:e^{it(s)\left(\lambda^2 + \frac{Q^2}{4}\right)}\:\hat{f}(\lambda)\: {|{\bf c}(\lambda)|}^{-2}\: d\lambda \\
&=& T_3f(s)\:+\: T_4f(s)\: + \:T_5f(s)\:.
\end{eqnarray*}

Estimates of $T_3$ and $T_4$ being similar, we give the details only for $T_3$. We note that there exists $\theta$ such that for all $0 < s \le R_0$, $|\theta (s)| =1$ and
\begin{eqnarray*}
&&|T_3f(s)| \\
&=&\theta(s)T_3f(s) \\
& =&c\:\theta(s) \bigintssss_{B(n)/s}^\infty {\left(\frac{s^{n-1}}{A(s)}\right)}^{1/2} \frac{1}{{(\lambda s)}^{\frac{n-1}{2}}} \: e^{i\left(\lambda s - \frac{\pi}{4}(n-1)\right)}  \:e^{it(s)\left(\lambda^2 + \frac{Q^2}{4}\right)}\:\hat{f}(\lambda)\: {|{\bf c}(\lambda)|}^{-2}\: d\lambda\:.
\end{eqnarray*}
Thus by Fubini's theorem and the Cauchy-Schwarz inequality,
\begin{eqnarray*}
&& \int_{0}^{R_0} |T_3f(s)|\:A(s)\:ds \\
&=& c \int_0^{R_0}\theta(s) \bigintssss_{B(n)/s}^\infty \lambda^{-\frac{n-1}{2}} e^{i\left(\lambda s - \frac{\pi}{4}(n-1)\right)}e^{it(s)\left(\lambda^2 + \frac{Q^2}{4}\right)}\hat{f}(\lambda){|{\bf c}(\lambda)|}^{-2} {A(s)}^{1/2}\: d\lambda \: ds \\
&=& c \int_{B(n)/R_0}^\infty \hat{f}(\lambda)\lambda^{-\frac{n-1}{2}} \left(\int_{B(n)/\lambda}^{R_0} \theta(s) {A(s)}^{1/2}  e^{i\left(\lambda s - \frac{\pi}{4}(n-1) + t(s)\left(\lambda^2 + \frac{Q^2}{4}\right)\right)} ds \right) {|{\bf c}(\lambda)|}^{-2} d\lambda \\
& \le & c \int_{B(n)/R_0}^\infty \left|\hat{f}(\lambda)\right|\lambda^{-\frac{n-1}{2}} \left|\int_{B(n)/\lambda}^{R_0} \theta(s) {A(s)}^{1/2}  e^{i\left(\lambda s + t(s)\left(\lambda^2 + \frac{Q^2}{4}\right)\right)} ds \right| {|{\bf c}(\lambda)|}^{-2} d\lambda \\
& \le &  {\left(\bigintssss_0^\infty {\left(\lambda^2 + \frac{Q^2}{4}\right)}^{1/4} {\left|\hat{f}(\lambda)\right|}^2 {|{\bf c}(\lambda)|}^{-2} d\lambda \right)}^{1/2} \\
&& \times {\left(\int_{B(n)/R_0}^\infty
\frac{{|{\bf c}(\lambda)|}^{-2}}{{\left(\lambda^2 + \frac{Q^2}{4}\right)}^{1/4} \lambda^{n-1}}
{\left|\bigintssss_{B(n)/\lambda}^{R_0} \theta(s) {A(s)}^{1/2}  e^{i\left(\lambda s + t(s)\left(\lambda^2 + \frac{Q^2}{4}\right)\right)} ds \right|}^2  d\lambda\right)}^{1/2} \\
&=& \: {\|f\|}_{H^{1/4}(S)} {\left(\int_{B(n)/R_0}^\infty G_1(\lambda) d\lambda\right)}^{1/2}\:.
\end{eqnarray*}
Now by the estimate of ${|{\bf c}(\lambda)|}^{-2}$ (see \ref{plancherel_measure}), Fubini's theorem and Lemma \ref{oscillatory_integral_estimate}, one has
\begin{eqnarray*}
&& \int_{B(n)/R_0}^\infty G_1(\lambda) d\lambda \\
&=& \int_{B(n)/R_0}^\infty
\frac{{|{\bf c}(\lambda)|}^{-2}}{{\left(\lambda^2 + \frac{Q^2}{4}\right)}^{1/4} \lambda^{n-1}}
{\left|\bigintssss_{B(n)/\lambda}^{R_0} \theta(s) {A(s)}^{1/2}  e^{i\left(\lambda s + t(s)\left(\lambda^2 + \frac{Q^2}{4}\right)\right)} ds \right|}^2  d\lambda \\
& \le & c \int_{B(n)/R_0}^\infty
\frac{1}{{\left(\lambda^2 + \frac{Q^2}{4}\right)}^{1/4}}
{\left|\bigintssss_{B(n)/\lambda}^{R_0} \theta(s) {A(s)}^{1/2}  e^{i\left(\lambda s + t(s)\left(\lambda^2 + \frac{Q^2}{4}\right)\right)} ds \right|}^2  d\lambda \\
&=& c \int_{B(n)/R_0}^\infty
\frac{1}{{\left(\lambda^2 + \frac{Q^2}{4}\right)}^{1/4}} \\
&& \times
\bigintssss_{B(n)/\lambda}^{R_0} \bigintssss_{B(n)/\lambda}^{R_0} \theta(s) \overline{\theta(s')} {\left(A(s)A(s')\right)}^{1/2}  e^{i\left\{\lambda (s-s') + (t(s)-t(s'))\left(\lambda^2 + \frac{Q^2}{4}\right)\right\}} ds ds'   d\lambda \\
&=& \int_0^{R_0} \theta(s) {A(s)}^{1/2} \int_0^{R_0} \overline{\theta(s')} {A(s')}^{1/2} \\
&& \times \left(\bigintssss_{\max\left\{\frac{B(n)}{s},\frac{B(n)}{s'}\right\}}^\infty \frac{e^{i\left\{\lambda(s-s')+(t(s)-t(s'))\left(\lambda^2 + \frac{Q^2}{4}\right)\right\}}}{{\left(\lambda^2 + \frac{Q^2}{4}\right)}^{1/4}} d\lambda\right) ds' ds \\
& \le & c \int_0^{R_0} {A(s)}^{1/2} \int_0^{R_0}  \frac{{A(s')}^{1/2}}{{|s-s'|}^{1/2}} \: ds' ds \\
& < & +\infty\:.
\end{eqnarray*}

Therefore,
\begin{equation} \label{thm_pf_eq9}
\int_{0}^{R_0} |T_3f(s)|\:A(s)\:ds \le c\: {\|f\|}_{H^{1/4}(S)}\:.
\end{equation}
Similarly,
\begin{equation} \label{thm_pf_eq10}
\int_{0}^{R_0} |T_4f(s)|\:A(s)\:ds \le c\: {\|f\|}_{H^{1/4}(S)}\:.
\end{equation}

We now estimate $T_5$. In this case, we ignore the oscillation and only use the estimate (\ref{thm_pf_eq8}) of the error term $\mathscr{E}_1$ . Combining the aforementioned estimate, the Cauchy-Schwarz inequality and the estimate of ${|{\bf c}(\lambda)|}^{-2}$ (\ref{plancherel_measure}), one has
\begin{eqnarray*}
|T_5f(s)| &\le & c \bigintssss_{B(n)/s}^\infty {(\lambda s)}^{-\frac{n+1}{2}} \:\left|\hat{f}(\lambda)\right|\: {|{\bf c}(\lambda)|}^{-2}\: d\lambda \\
& = & \frac{c}{s^{\frac{n+1}{2}}} \bigintssss_{B(n)/s}^\infty \lambda ^{-\frac{n+1}{2}} \:\left|\hat{f}(\lambda)\right|\: {|{\bf c}(\lambda)|}^{-2}\: d\lambda \\
& \le & \frac{c}{s^{\frac{n+1}{2}}} {\left(\int_0^\infty {\left(\lambda^2 + \frac{Q^2}{4}\right)}^{1/4} {\left|\hat{f}(\lambda)\right|}^2 {|{\bf c}(\lambda)|}^{-2}\: d\lambda \right)}^{1/2} \\
&& \times {\left(\int_{B(n)/s}^\infty \frac{\lambda^{-(n+1)}}{{\left(\lambda^2 + \frac{Q^2}{4}\right)}^{1/4}} {|{\bf c}(\lambda)|}^{-2}\: d\lambda \right)}^{1/2} \\
& \le & \frac{c}{s^{\frac{n+1}{2}}} {\|f\|}_{H^{1/4}(S)} {\left(\int_{B(n)/s}^\infty  \frac{d\lambda}{\lambda^2{\left(\lambda^2 + \frac{Q^2}{4}\right)}^{1/4}} \right)}^{1/2} \\
& \le & \frac{c}{s^{\frac{n+1}{2}}} {\|f\|}_{H^{1/4}(S)} {\left(\int_{B(n)/s}^\infty  \lambda^{-\frac{5}{2}} d\lambda \right)}^{1/2} \\
& = & \frac{c}{s^{\left(\frac{2n-1}{4}\right)}} {\|f\|}_{H^{1/4}(S)} \:.
\end{eqnarray*}
Hence,
\begin{eqnarray} \label{thm_pf_eq11}
\int_{0}^{R_0} |T_5f(s)|\:A(s)\:ds &\le & c\: {\|f\|}_{H^{1/4}(S)} \int_{0}^{R_0} s^{\left(\frac{2n-3}{4}\right)} \: ds \nonumber \\
&\le & c\: {\|f\|}_{H^{1/4}(S)} \:.
\end{eqnarray}
Thus combining (\ref{thm_pf_eq9})-(\ref{thm_pf_eq11}), we get that
\begin{equation*}
I_1 \le c\: {\|f\|}_{H^{1/4}(S)} \:.
\end{equation*}
\subsubsection{Estimating $I_2$:} The strategy to estimate $I_2$ is essentially the same as that of $I_1$. This time around one uses the  series expansion (\ref{harish_chandra_expansion}) of $\varphi_\lambda$.

\medskip

For $R_0 \le s \le R$ and $\lambda \ge B(n)/s$, using the series expansion  (\ref{harish_chandra_expansion}) and the estimate (\ref{coefficient_estimate}) on the coefficients $\Gamma_\mu$, we get
\begin{equation} \label{thm_pf_eq12}
\varphi_\lambda (s) = 2^{-m_\z/2} {A(s)}^{-1/2} \left\{{\bf c}(\lambda)  e^{i \lambda s} + {\bf c}(-\lambda) e^{-i\lambda  s}\right\} + \mathscr{E}_2(\lambda,s)\:,
\end{equation}
where
\begin{equation*}
\mathscr{E}_2(\lambda,s) = 2^{-m_\z/2} {A(s)}^{-1/2} \left\{{\bf c}(\lambda)  \displaystyle\sum_{\mu=1}^\infty \Gamma_\mu(\lambda) e^{(i \lambda-\mu) s} + {\bf c}(-\lambda) \displaystyle\sum_{\mu=1}^\infty \Gamma_\mu(-\lambda) e^{-(i\lambda + \mu) s}\right\} \:,
\end{equation*}
and thus
\begin{equation} \label{thm_pf_eq13}
\left|\mathscr{E}_2(\lambda,s)\right| \le c \: {A(s)}^{-1/2} \left|{\bf c}(\lambda)\right| {(1+|\lambda|)}^{-1} \:.
\end{equation}

Hence for $R_0 \le s \le R$ and $\lambda \ge B(n)/s$, in accordance to (\ref{thm_pf_eq12}), we decompose $T_2$ as,
\begin{eqnarray*}
T_2f(s)&=& 2^{-m_\z/2} {A(s)}^{-1/2} \bigintssss_{B(n)/s}^\infty  {\bf c}(\lambda) \: e^{i\left\{\lambda s + t(s)\left(\lambda^2 + \frac{Q^2}{4}\right)\right\}} \:\hat{f}(\lambda)\: {|{\bf c}(\lambda)|}^{-2}\: d\lambda \\
&+& 2^{-m_\z/2} {A(s)}^{-1/2} \bigintssss_{B(n)/s}^\infty  {\bf c}(-\lambda) \: e^{i\left\{-\lambda s + t(s)\left(\lambda^2 + \frac{Q^2}{4}\right)\right\}} \:\hat{f}(\lambda)\: {|{\bf c}(\lambda)|}^{-2}\: d\lambda \\
&+& \bigintssss_{B(n)/s}^\infty \mathscr{E}_2(\lambda,s) \:e^{it(s)\left(\lambda^2 + \frac{Q^2}{4}\right)}\:\hat{f}(\lambda)\: {|{\bf c}(\lambda)|}^{-2}\: d\lambda \\
&=& T_6f(s)\:+\: T_7f(s)\: + \:T_8f(s)\:.
\end{eqnarray*}

Estimates of $T_6$ and $T_7$ being similar, we give the details only for $T_6$. There exists $\psi$ such that for all $R_0 \le s \le R$, $|\psi (s)| =1$ and
\begin{eqnarray*}
&&|T_6f(s)| \\
&=&\psi(s)\: T_6f(s) \\
& =& 2^{-m_\z/2} \:\psi(s) {A(s)}^{-1/2} \bigintssss_{B(n)/s}^\infty  {\bf c}(\lambda) \: e^{i\left\{\lambda s + t(s)\left(\lambda^2 + \frac{Q^2}{4}\right)\right\}} \:\hat{f}(\lambda)\: {|{\bf c}(\lambda)|}^{-2}\: d\lambda\:.
\end{eqnarray*}
Hence by Fubini's theorem,
\begin{eqnarray*}
&& \int_{R_0}^R \left|T_6f(s)\right| \: A(s)\: ds \\
&=& 2^{-m_\z/2} \: \int_{R_0}^R  \psi(s)\:  {A(s)}^{-1/2} \bigintssss_{B(n)/s}^\infty  {\bf c}(\lambda) \: e^{i\left\{\lambda s + t(s)\left(\lambda^2 + \frac{Q^2}{4}\right)\right\}} \:\hat{f}(\lambda)\: {|{\bf c}(\lambda)|}^{-2}\: A(s)\: d\lambda \: ds  \\
& \le & \int_{B(n)/R}^\infty  \left|\hat{f}(\lambda)\right| \: {|{\bf c}(\lambda)|}^{-1} \left|\bigintssss_{\max\left\{\frac{B(n)}{\lambda}, R_0\right\}}^R \psi(s) A(s)^{1/2}\: e^{i\left\{\lambda s + t(s)\left(\lambda^2 + \frac{Q^2}{4}\right)\right\}} \: ds  \right| d\lambda \\
&=& \int_{B(n)/R}^{B(n)/R_0}  \left|\hat{f}(\lambda)\right| \: {|{\bf c}(\lambda)|}^{-1} \left|\bigintssss_{B(n)/ \lambda}^R \psi(s) A(s)^{1/2}\: e^{i\left\{\lambda s + t(s)\left(\lambda^2 + \frac{Q^2}{4}\right)\right\}} \: ds  \right| d\lambda \\
&& +  \int_{B(n)/R_0}^\infty  \left|\hat{f}(\lambda)\right| \: {|{\bf c}(\lambda)|}^{-1} \left|\bigintssss_{R_0}^R \psi(s) A(s)^{1/2}\: e^{i\left\{\lambda s + t(s)\left(\lambda^2 + \frac{Q^2}{4}\right)\right\}} \: ds  \right| d\lambda \\
&=& I_3+I_4 \:.
\end{eqnarray*}

We first estimate $I_3$. By the Cauchy-Schwartz inequality,

\begin{eqnarray*}
 I_3 &=& \int_{B(n)/R}^{B(n)/R_0}  \left|\hat{f}(\lambda)\right| \: {|{\bf c}(\lambda)|}^{-1} \left|\bigintssss_{B(n)/ \lambda}^R \psi(s) A(s)^{1/2}\: e^{i\left\{\lambda s + t(s)\left(\lambda^2 + \frac{Q^2}{4}\right)\right\}} \: ds  \right| d\lambda \\
& \le & \int_{B(n)/R}^\infty  \left|\hat{f}(\lambda)\right| \: {|{\bf c}(\lambda)|}^{-1} \left|\bigintssss_{B(n)/ \lambda}^R \psi(s) A(s)^{1/2}\: e^{i\left\{\lambda s + t(s)\left(\lambda^2 + \frac{Q^2}{4}\right)\right\}} \: ds  \right| d\lambda \\
& \le & {\left(\int_0^\infty {\left(\lambda^2 + \frac{Q^2}{4}\right)}^{1/4} \:{\left|\hat{f}(\lambda)\right|}^2 \: {|{\bf c}(\lambda)|}^{-2}\: d\lambda \right)}^{1/2} \\
&& \times {\left(\int_{B(n)/R}^\infty \frac{1}{{\left(\lambda^2 + \frac{Q^2}{4}\right)}^{1/4}} {\left|\bigintssss_{B(n)/\lambda}^R \psi(s) A(s)^{1/2}\: e^{i\left\{\lambda s + t(s)\left(\lambda^2 + \frac{Q^2}{4}\right)\right\}} \: ds  \right|}^2 d\lambda \right)}^{1/2} \\
&=& {\|f\|}_{H^{1/4}(S)} {\left(\int_{B(n)/R}^\infty G_2(\lambda) d\lambda\right)}^{1/2} \:.
\end{eqnarray*}
Again by Fubini's theorem and Lemma \ref{oscillatory_integral_estimate} we have,
\begin{eqnarray*}
&& \int_{B(n)/R}^\infty G_2(\lambda) d\lambda \\
&=& \int_{B(n)/R}^\infty \frac{1}{{\left(\lambda^2 + \frac{Q^2}{4}\right)}^{1/4}} {\left|\bigintssss_{B(n)/\lambda}^R \psi(s) A(s)^{1/2}\: e^{i\left\{\lambda s + t(s)\left(\lambda^2 + \frac{Q^2}{4}\right)\right\}} \: ds  \right|}^2 d\lambda \\
&=& \int_{B(n)/R}^\infty \frac{1}{{\left(\lambda^2 + \frac{Q^2}{4}\right)}^{1/4}} \\
&& \times \bigintssss_{B(n)/\lambda}^R \bigintssss_{B(n)/\lambda}^R \psi(s) \overline{\psi(s')}\: {\left(A(s)A(s')\right)}^{1/2} \:e^{i\left\{\lambda (s-s') + (t(s)-t(s'))\left(\lambda^2 + \frac{Q^2}{4}\right)\right\}} \: ds \: ds' \: d\lambda \\
&=& \int_0^R \psi(s)  {A(s)}^{1/2} \int_0^R \overline{\psi(s')} {A(s')}^{1/2} \\
&& \times \left(\bigintssss_{\max\left\{\frac{B(n)}{s},\frac{B(n)}{s'}\right\}}^\infty \frac{e^{i\left\{\lambda(s-s')+(t(s)-t(s'))\left(\lambda^2 + \frac{Q^2}{4}\right)\right\}}}{{\left(\lambda^2 + \frac{Q^2}{4}\right)}^{1/4}} d\lambda\right) ds'ds \\
& \le & c\:  \int_0^R {A(s)}^{1/2} \int_0^R \frac{{A(s')}^{1/2}}{{|s-s'|}^{1/2}} \: ds'ds  \\
& < &  +\infty\:.
\end{eqnarray*}
Therefore,
\begin{equation} \label{thm_pf_eq14}
I_3 \le c \: {\|f\|}_{H^{1/4}(S)}\:.
\end{equation}

Next we estimate $I_4$. By the Cauchy-Schwartz inequality,
\begin{eqnarray*}
 I_4 &=& \int_{B(n)/R_0}^\infty  \left|\hat{f}(\lambda)\right| \: {|{\bf c}(\lambda)|}^{-1} \left|\bigintssss_{R_0}^R \psi(s) A(s)^{1/2}\: e^{i\left\{\lambda s + t(s)\left(\lambda^2 + \frac{Q^2}{4}\right)\right\}} \: ds  \right| d\lambda \\
& \le & {\left(\int_0^\infty {\left(\lambda^2 + \frac{Q^2}{4}\right)}^{1/4} \:{\left|\hat{f}(\lambda)\right|}^2 \: {|{\bf c}(\lambda)|}^{-2}\: d\lambda \right)}^{1/2} \\
&& \times {\left(\int_{B(n)/R_0}^\infty \frac{1}{{\left(\lambda^2 + \frac{Q^2}{4}\right)}^{1/4}} {\left|\bigintssss_{R_0}^R \psi(s) A(s)^{1/2}\: e^{i\left\{\lambda s + t(s)\left(\lambda^2 + \frac{Q^2}{4}\right)\right\}} \: ds  \right|}^2 d\lambda \right)}^{1/2} \\
&=& {\|f\|}_{H^{1/4}(S)} {\left(\int_{B(n)/R_0}^\infty G_3(\lambda) d\lambda\right)}^{1/2} \:.
\end{eqnarray*}

Again by Fubini's theorem we have,
\begin{eqnarray*}
&& \int_{B(n)/R_0}^\infty G_3(\lambda) d\lambda \\
&=& \int_{B(n)/R_0}^\infty \frac{1}{{\left(\lambda^2 + \frac{Q^2}{4}\right)}^{1/4}} {\left|\bigintssss_{R_0}^R \psi(s) A(s)^{1/2}\: e^{i\left\{\lambda s + t(s)\left(\lambda^2 + \frac{Q^2}{4}\right)\right\}} \: ds  \right|}^2 d\lambda \\
&=& \int_{B(n)/R_0}^\infty \frac{1}{{\left(\lambda^2 + \frac{Q^2}{4}\right)}^{1/4}} \\
&& \times \bigintssss_{R_0}^R \bigintssss_{R_0}^R \psi(s) \overline{\psi(s')}\: {\left(A(s)A(s')\right)}^{1/2} \:e^{i\left\{\lambda (s-s') + (t(s)-t(s'))\left(\lambda^2 + \frac{Q^2}{4}\right)\right\}} \: ds \: ds' \: d\lambda \\
&=& \int_{R_0}^R \psi(s)  {A(s)}^{1/2} \int_{R_0}^R \overline{\psi(s')} {A(s')}^{1/2} \\
&& \times \left(\bigintssss_{\frac{B(n)}{R_0}}^\infty \frac{e^{i\left\{\lambda(s-s')+(t(s)-t(s'))\left(\lambda^2 + \frac{Q^2}{4}\right)\right\}}}{{\left(\lambda^2 + \frac{Q^2}{4}\right)}^{1/4}} d\lambda\right) ds'ds \:.
\end{eqnarray*}
Now as $R_0 \le s \le R$, we have
\begin{equation*}
\frac{B(n)}{R} \le \frac{B(n)}{s} \le \frac{B(n)}{R_0}\:,
\end{equation*}
and similarly for $s'$. Thus in particular,
\begin{equation} \label{thm_pf_eq15}
\frac{B(n)}{R} \le \max\left\{\frac{B(n)}{s},\frac{B(n)}{s'}\right\} \le \frac{B(n)}{R_0}\:.
\end{equation}
Next we note that by Lemma \ref{oscillatory_integral_estimate},
\begin{equation*}
\left|\bigintssss_{\max\left\{\frac{B(n)}{s},\frac{B(n)}{s'}\right\}}^\infty \frac{e^{i\left\{\lambda(s-s')+(t(s)-t(s'))\left(\lambda^2 + \frac{Q^2}{4}\right)\right\}}}{{\left(\lambda^2 + \frac{Q^2}{4}\right)}^{1/4}} d\lambda \right| \le  \frac{c}{{|s-s'|}^{1/2}}\:.
\end{equation*}
By (\ref{thm_pf_eq15}), we also have
\begin{eqnarray*}
\left| \bigintssss_{\max\left\{\frac{B(n)}{s},\frac{B(n)}{s'}\right\}}^{\frac{B(n)}{R_0}} \frac{e^{i\left\{\lambda(s-s')+(t(s)-t(s'))\left(\lambda^2 + \frac{Q^2}{4}\right)\right\}}}{{\left(\lambda^2 + \frac{Q^2}{4}\right)}^{1/4}} d\lambda \right| & \le & \bigintssss_{\max\left\{\frac{B(n)}{s},\frac{B(n)}{s'}\right\}}^{\frac{B(n)}{R_0}} \frac{1}{\lambda^{1/2}}\: d\lambda \\
& \le &  \bigintssss_{\frac{B(n)}{R}}^{\frac{B(n)}{R_0}} \frac{1}{\lambda^{1/2}} \: d\lambda \:.
\end{eqnarray*}
Thus it is justified to decompose the following integral as
\begin{eqnarray*}
&&\bigintssss_{\frac{B(n)}{R_0}}^\infty \frac{e^{i\left\{\lambda(s-s')+(t(s)-t(s'))\left(\lambda^2 + \frac{Q^2}{4}\right)\right\}}}{{\left(\lambda^2 + \frac{Q^2}{4}\right)}^{1/4}} d\lambda \\
&=& \bigintssss_{\max\left\{\frac{B(n)}{s},\frac{B(n)}{s'}\right\}}^\infty \frac{e^{i\left\{\lambda(s-s')+(t(s)-t(s'))\left(\lambda^2 + \frac{Q^2}{4}\right)\right\}}}{{\left(\lambda^2 + \frac{Q^2}{4}\right)}^{1/4}} d\lambda \\
&& -  \bigintssss_{\max\left\{\frac{B(n)}{s},\frac{B(n)}{s'}\right\}}^{\frac{B(n)}{R_0}} \frac{e^{i\left\{\lambda(s-s')+(t(s)-t(s'))\left(\lambda^2 + \frac{Q^2}{4}\right)\right\}}}{{\left(\lambda^2 + \frac{Q^2}{4}\right)}^{1/4}} d\lambda \:.
\end{eqnarray*}
Hence,
\begin{eqnarray*}
&&\left|\bigintssss_{\frac{B(n)}{R_0}}^\infty \frac{e^{i\left\{\lambda(s-s')+(t(s)-t(s'))\left(\lambda^2 + \frac{Q^2}{4}\right)\right\}}}{{\left(\lambda^2 + \frac{Q^2}{4}\right)}^{1/4}} d\lambda \right|\\
&\le & \left|\bigintssss_{\max\left\{\frac{B(n)}{s},\frac{B(n)}{s'}\right\}}^\infty \frac{e^{i\left\{\lambda(s-s')+(t(s)-t(s'))\left(\lambda^2 + \frac{Q^2}{4}\right)\right\}}}{{\left(\lambda^2 + \frac{Q^2}{4}\right)}^{1/4}} d\lambda \right| \\
&& + \left| \bigintssss_{\max\left\{\frac{B(n)}{s},\frac{B(n)}{s'}\right\}}^{\frac{B(n)}{R_0}} \frac{e^{i\left\{\lambda(s-s')+(t(s)-t(s'))\left(\lambda^2 + \frac{Q^2}{4}\right)\right\}}}{{\left(\lambda^2 + \frac{Q^2}{4}\right)}^{1/4}} d\lambda \right| \\
& \le & c \left( \frac{1}{{|s-s'|}^{1/2}} + 1\right)\:.
\end{eqnarray*}
 So we get that
\begin{equation*}
\int_{B(n)/R_0}^\infty G_3(\lambda) d\lambda \le c\: \int_{R_0}^R {A(s)}^{1/2} \int_{R_0}^R {A(s')}^{1/2} \left( \frac{1}{{|s-s'|}^{1/2}} + 1\right) \:ds'\:ds< +\infty\:.
\end{equation*}
Then it follows that
\begin{equation} \label{thm_pf_eq16}
I_4 \le c \: {\|f\|}_{H^{1/4}(S)}\:.
\end{equation}
Therefore, combining (\ref{thm_pf_eq14}) and (\ref{thm_pf_eq16}), we get that \begin{equation} \label{thm_pf_eq17}
\int_{R_0}^R \left|T_6f(s)\right|\: A(s)\: ds \le c \: {\|f\|}_{H^{1/4}(S)}\:.
\end{equation}
Similarly,
\begin{equation} \label{thm_pf_eq18}
\int_{R_0}^R \left|T_7f(s)\right|\: A(s)\: ds \le c \: {\|f\|}_{H^{1/4}(S)}\:.
\end{equation}

\medskip

Finally, we estimate $T_8$ by using the estimate (\ref{thm_pf_eq13}) on the error term $\mathscr{E}_2$. Combining the aforementioned estimate, the Cauchy-Schwarz inequality and the estimate of ${|{\bf c}(\lambda)|}^{-2}$ (\ref{plancherel_measure}), one has
\begin{eqnarray*}
\left|T_8f(s)\right| & \le & c \: {A(s)}^{-1/2}\int_{B(n)/s}^\infty |{\bf c}(\lambda)| {(1+|\lambda|)}^{-1} \left|\hat{f}(\lambda)\right| {|{\bf c}(\lambda)|}^{-2} d\lambda \\
& \le & c\: \: {A(s)}^{-1/2} {\left(\int_0^\infty {\left(\lambda^2 + \frac{Q^2}{4}\right)}^{1/4} {\left|\hat{f}(\lambda)\right|}^2 {|{\bf c}(\lambda)|}^{-2}\: d\lambda \right)}^{1/2} {\left(\int_{B(n)/s}^\infty \frac{d\lambda}{\lambda^{5/2}}\right)}^{1/2} \\
& = & c \: {A(s)}^{-1/2}\: s^{3/4} \: {\|f\|}_{H^{1/4}(S)}\:.
\end{eqnarray*}

Hence
\begin{eqnarray} \label{thm_pf_eq19}
\int_{R_0}^R \left|T_8f(s)\right|\: A(s)\: ds &\le & c \: {\|f\|}_{H^{1/4}(S)} \int_{R_0}^R \: s^{3/4} \: {A(s)}^{1/2}\: ds \nonumber \\
& \le & c \: {\|f\|}_{H^{1/4}(S)} \:.
\end{eqnarray}
So from (\ref{thm_pf_eq17})-(\ref{thm_pf_eq19}), we get
\begin{equation*}
I_2 \le c \: {\|f\|}_{H^{1/4}(S)} \:.
\end{equation*}
This completes the proof of Theorem \ref{maximal_bddness_thm}.

\section{A counter-example for lower regularity}
In this section we make use of several standard notions in Harmonic Analysis on Riemannian Symmetric spaces. For unexplained terminologies and more details, we refer to \cite{Helgason}.

\medskip

Let $G=SL(2,\C)$ and $K$ be its maximal compact subgroup $SU(2)$. Here,
\begin{equation*}
A = \left\{a_t=
\begin{pmatrix}
e^t &0 \\
0 &e^{-t}
\end{pmatrix} : t \in \R
\right\}\:,
\end{equation*}
and $\Sigma_+$ consists of a single root $\alpha$ that occurs with multiplicity $2$.  We normalize $\alpha$ so that $\alpha(log\: a_t)=t$. Every $\lambda \in \C$ can be identified with an element in $\mathfrak{a}^*_{\C}$ by $\lambda = \lambda \alpha$. We see that in this identification (the half-sum of positive roots counted with multiplicity) $\rho=1,\: Q=2\rho=2$ and
\begin{equation*}
G/K \cong \mathbb{H}^3(-1)\:,
\end{equation*}
where $\mathbb{H}^3(-1)$ is the $3$-dimensional Real Hyperbolic space with constant sectional curvature $-1$. Then by \cite[p. 432, Theorem 5.7]{Helgason} we get the following information for $G/K$:
\begin{itemize}
\item For $\lambda \in (0,\infty)$, the Spherical function $\varphi_\lambda$ is given by,
\begin{equation*}
\varphi_\lambda(s) = \frac{\sin(\lambda s)}{\lambda \sinh(s)}\:.
\end{equation*}
\item $G$-invariant Riemannian volume measure: $\sinh^2 s \: ds \: dk$, where $ds$ and $dk$ are the Lebesgue measure on $(0,\infty)$ and the Haar measure on $SU(2)$ respectively.
\item Plancherel measure: $\lambda^2 \: d\lambda$, where $d\lambda$ is the Lebesgue measure on $(0,\infty)$.
\end{itemize}

By the well-known properties of the Abel transform on Ch\'ebli-Trim\`eche Hypergroups \cite{BX} (for Symmetric Spaces see \cite{Anker} and also \cite[Lemma 2.3]{Helgason-b}), we have the following commutative diagram, where
\begin{itemize}
\item $\mathscr{A}_1$ is the Abel transform defined from the space of $K$-biinvariant $L^2$-Schwartz class functions $\mathscr{S}^2(G//K)$ onto the space of even Schwartz class functions on $\R$, ${\mathscr{S}(\R)}_{even}$\:.
\item $\mathscr{A}_2$ is the Abel transform defined from the space of radial Schwartz class functions on $\R^3$, $\mathscr{S}(\R^3)_{radial}$ onto ${\mathscr{S}(\R)}_{even}$\:.
\item $\wedge$ denotes the Spherical Fourier transform from $\mathscr{S}^2(G//K)$ onto ${\mathscr{S}(\R)}_{even}$\:.
\item $\mathscr{F}$ denotes the Euclidean Spherical Fourier transform from $\mathscr{S}(\R^3)_{radial}$ onto ${\mathscr{S}(\R)}_{even}$\:.
\item $\sim$ denotes the 1-dimensional Euclidean Fourier transform from ${\mathscr{S}(\R)}_{even}$ onto itself.
\end{itemize}

\[
\begin{tikzcd}[row sep=1.4cm,column sep=1.4cm]
\mathscr{S}^2(G//K)\arrow[r,"\mathscr{A}_1"] \arrow[dr,swap,"\wedge"] & {\mathscr{S}(\R)}_{even}
\arrow[d,"\sim"] &
\arrow[l,"\mathscr{A}_2",swap]  \mathscr{S}(\R^3)_{radial} \arrow[dl,"\mathscr{F}"] \\
&  {\mathscr{S}(\R)}_{even}&
\end{tikzcd}
\]

Then, since all the maps above are topological isomorphisms, defining $\mathscr{A}:=\mathscr{A}^{-1}_2 \circ \mathscr{A}_1$, one can reduce matters to the following simplified commutative diagram:

\[
\begin{tikzcd}[row sep=1.4cm,column sep=1.4cm]
\mathscr{S}^2(G//K)\arrow[r,"\mathscr{A}"] \arrow[dr,swap,"\wedge"] & \mathscr{S}(\R^3)_{radial}
\arrow[d,"\mathscr{F}"]   \\
&  {\mathscr{S}(\R)}_{even}&
\end{tikzcd}
\]

Hence, for $f \in \mathscr{S}^2(G//K)$, we get
\begin{equation} \label{abel_rel}
\hat{f}(\lambda)= \mathscr{F}(\mathscr{A}f)(\lambda)\:.
\end{equation}

We now present the proof of Theorem \ref{sharpness_thm}.
\begin{proof}[Proof of Theorem \ref{sharpness_thm}]
Let $f \in \mathscr{S}^2(G//K)$. Then by the preceding arguments, $\mathscr{A}f \in \mathscr{S}(\R^3)_{radial}$. We first show that the maximal function corresponding to the solution of the Schr\"odinger equation with initial data $f$ on $G/K$ is comparable to its Euclidean counterpart with $\mathscr{A}f$ as the initial data. Indeed for any $t>0$, using (\ref{abel_rel}) we have,
\begin{eqnarray*}
S_t f(a_s) &=& \int_0^\infty \varphi_\lambda (a_s) \: e^{it(\lambda^2+1)} \: \hat{f}(\lambda) \: \lambda^2 \: d\lambda \\
&=& \int_0^\infty \frac{\sin(\lambda s)}{\lambda \sinh(s)} \: e^{it(\lambda^2+1)} \: \mathscr{F}(\mathscr{A}f)(\lambda) \: \lambda^2 \: d\lambda \\
&=& \left(\frac{e^{it} s}{\sinh(s)}\right) \int_0^\infty \frac{\sin(\lambda s)}{\lambda s} \: e^{it\lambda^2} \: \mathscr{F}(\mathscr{A}f)(\lambda) \: \lambda^2 \: d\lambda \\
&=& \left(\frac{e^{it} s}{\sinh(s)}\right) \int_0^\infty \left(\sqrt{2}\: \Gamma\left(\frac{3}{2}\right) \frac{J_{\frac{1}{2}}(\lambda s)}{{(\lambda s)}^{\frac{1}{2}}} \right)\: e^{it\lambda^2} \: \mathscr{F}(\mathscr{A}f)(\lambda) \: \frac{\lambda^2 \: d\lambda}{\sqrt{2}\: \Gamma(\frac{3}{2})} \\
&=& \left(\frac{e^{it} s}{\sinh(s)}\right) \tilde{S}_t (\mathscr{A}f)(s)\:,
\end{eqnarray*}
where $\tilde{S}_t (\mathscr{A}f)$ is the solution of the Schr\"odinger equation with initial data $\mathscr{A}f$ on $\R^3$. Thus we note that for some positive constant $c$,
\begin{eqnarray} \label{equality_1}
{\left\|\displaystyle\sup_{0<t<1} \left| S_t f(\cdot)\right|\right\|}_{L^1(B_{100})} &=& \int_0^{100} \displaystyle\sup_{0<t<1} \left| S_t f(a_s)\right| \: \sinh^2s\:ds \nonumber\\
& \asymp &  \int_0^{100} \displaystyle\sup_{0<t<1} \left|\tilde{S}_t (\mathscr{A}f)(s)\right| \: s^2\:ds \nonumber\\
&=& c \:{\left\|\displaystyle\sup_{0<t<1} \left| \tilde{S}_t (\mathscr{A}f)(\cdot)\right|\right\|}_{L^1(B(o,100))} \:.
\end{eqnarray}

\medskip

For $\beta>0$, we next consider the inhomogeneous Sobolev spaces of radial functions on $\R^3$ \cite[p. 135]{Prestini}:
\begin{equation*}
H^{\beta}(\R^3)_{radial} = \left\{g \in L^2(\R^3)_{radial}: \int_{\R^3} {\left(1+{|\xi|}^2\right)}^\beta \:{\left|\mathscr{F}g(\xi)\right|}^2\: d\xi < \infty\right\}\:.
\end{equation*}

Then again using (\ref{abel_rel}), we get that there exists  a positive constant $c$ such that for any $\beta >0$,
\begin{eqnarray} \label{equality_2}
{\|f\|}_{H^\beta(G/K)} &=& {\left(\int_0^\infty {(\lambda^2 + 1)}^\beta \:{\left|\hat{f}(\lambda)\right|}^2 \: \lambda^2 \: d\lambda\right)}^{1/2} \nonumber\\
&=& {\left(\int_0^\infty {(\lambda^2 + 1)}^\beta \:{\left|\mathscr{F}(\mathscr{A}f)(\lambda)\right|}^2 \: \lambda^2 \: d\lambda\right)}^{1/2} \nonumber\\
&=& c \: {\|\mathscr{A}f\|}_{H^\beta(\R^3)}\:.
\end{eqnarray}

\medskip

Finally we choose and fix $\beta \in (0, 1/4)$. We note by \cite[Theorem 1]{Prestini} that there exists a radial Schwartz class function $g$ on $\R^3$ that fails the maximal inequality:
\begin{equation*}
\:{\left\|\displaystyle\sup_{0<t<1} \left| \tilde{S}_t g(\cdot)\right|\right\|}_{L^1(B(o,100))} \le c_n {\|g\|}_{H^\beta(\R^3)}\:,
\end{equation*}
for all $c_n>0$. Then by (\ref{equality_1}) and (\ref{equality_2}) it follows that, $\mathscr{A}^{-1}g$ is a radial $L^2$-Schwartz class function on $G/K$ that fails the maximal inequality (\ref{maximal_bddness_inequality}) for the geodesic ball $B_{100}$.
\end{proof}

\section{Concluding remarks}
In this section, we make some remarks and pose some new problems.
\subsection{The fractional Schr\"odinger equation:}
\begin{enumerate}
\item For $a>1$, let us consider the fractional Schr\"odinger equation on $S$ given by
\begin{equation} \label{frac_schrodinger}
\begin{cases}
	 i\frac{\partial u}{\partial t} =-{(-\Delta)}^{a/2} u\:,\:  (x,t) \in S \times \R \\
	u(0,\cdot)=f\:,\: \text{ on } S \:.
	\end{cases}
\end{equation}
Note that in Theorem \ref{maximal_bddness_thm} and Corollary \ref{pointwise_conv_cor}, we have only considered the case $a=2$. Now for a radial function $f$ belonging to the $L^2$-Schwartz class, the solution to (\ref{frac_schrodinger}) is given by
\begin{equation*}
S_t f(x):= \int_{0}^\infty \varphi_\lambda(x)\:e^{it{\left(\lambda^2 + \frac{Q^2}{4}\right)}^{a/2}}\:\hat{f}(\lambda)\: {|{\bf c}(\lambda)|}^{-2}\: d\lambda\:.
\end{equation*}
Then considering the corresponding maximal function and proceeding verbatim as in the proof of Theorem \ref{maximal_bddness_thm}, one can reduce matters to proving an analogue of Lemma \ref{oscillatory_integral_estimate}. More precisely, under the hypothesis of Lemma \ref{oscillatory_integral_estimate}, one needs to get an estimate of the form:
\begin{equation} \label{frac_oscillatory_integral_inequality}
\left|\bigintssss_{\max\left\{\frac{B(n)}{s},\frac{B(n)}{s'}\right\}}^\infty \frac{e^{i\left\{\lambda(s-s')+(t(s)-t(s')){\left(\lambda^2 + \frac{Q^2}{4}\right)}^{a/2}\right\}}}{{\left(\lambda^2 + \frac{Q^2}{4}\right)}^{1/4}} d\lambda \right| \le \frac{c}{{|s-s'|}^{1/2}}\:.
\end{equation}
In this setup, the modified phase function is of the form
\begin{equation*}
\psi(\lambda):= e^{i\left\{v \lambda \pm {(\lambda^2 +1)}^{a/2}\right\}}\:.
\end{equation*}
Now explicitly locating the stationary point of the above phase function, unlike in the case $a=2$, does not seem to be an easy task and thus obtaining (\ref{frac_oscillatory_integral_inequality}) becomes technically challenging. This difficulty in principle, can be attributed to the spectral gap of the Laplace-Beltrami operator, an intrinsic property of non-flat, noncompact spaces such as Damek-Ricci spaces.

\medskip

\item However, if one considers the fractional
Schr\"odinger equation corresponding to the shifted Laplace-Beltrami operator $\tilde{\Delta}:=\Delta + (Q^2/4)$ on $S$, that is, for $a>1$,
\begin{equation} \label{frac_schrodinger'}
\begin{cases}
	 i\frac{\partial u}{\partial t} =-{(-\tilde{\Delta})}^{a/2} u\:,\:  (x,t) \in S \times \R \\
	u(0,\cdot)=f\:,\: \text{ on } S \:,
	\end{cases}	
\end{equation}
then for a radial function $f$ belonging to the $L^2$-Schwartz class, the solution to (\ref{frac_schrodinger'}) is given by
\begin{equation*}
S_t f(x):= \int_{0}^\infty \varphi_\lambda(x)\:e^{it\lambda^a} \:\hat{f}(\lambda)\: {|{\bf c}(\lambda)|}^{-2}\: d\lambda\:.
\end{equation*}
In this case, working with the Sobolev spaces defined in (\ref{sobolev_space_defn}) or the inhomogeneous Sobolev spaces corresponding to $\tilde{\Delta}$,
\begin{equation*}
\tilde{H}^\alpha(S):=\left\{f \in L^2(S): {\|f\|}_{\tilde{H}^\alpha(S)}:= {\left(\int_0^\infty {\left(1+ \lambda^2 \right)}^\alpha \: {|\hat{f}(\lambda)|}^2 \: {|{\bf c}(\lambda)|}^{-2}\: d\lambda\right)}^{1/2}< \infty\right\}\:,
\end{equation*}
the proofs of Lemma \ref{oscillatory_integral_estimate} and Theorems  \ref{maximal_bddness_thm} and \ref{sharpness_thm} can be carried out verbatim to get obvious analogues of Theorem \ref{maximal_bddness_thm}, Corollary \ref{pointwise_conv_cor} and Theorem \ref{sharpness_thm} for the fractional operator (\ref{frac_schrodinger'}), for all $a>1$.
\end{enumerate}
\subsection{Generalization of Theorem \ref{sharpness_thm}:} In the special case of $\mathbb{H}^3 \cong SL(2,\C)/SU(2)$, using the similarity of the spherical function with its Euclidean counterpart, we could essentially reduce matters to $\R^3$. This is not the case for a general Damek-Ricci space however. Now explicit constructions of counter-examples in the literature (for instance \cite{DK, Sjolin, Prestini}) crucially use translations, dilations and the explicit knowledge of their interactions with the Fourier transform. Since these Euclidean features do not hold in the great generality of Damek-Ricci spaces, construction of counter-examples for lower regularity is non-trivial and is an interesting problem on its own.

\section*{Acknowledgements} The author is thankful to Prof. Swagato K. Ray for comments and suggestions. The author is supported by a Research Fellowship of Indian Statistical Institute.

\bibliographystyle{amsplain}

\end{document}